\documentclass[11pt,a4paper]{article}
\usepackage[dvips]{color}
\usepackage{indentfirst,latexsym,bm}
\usepackage{graphicx}
\usepackage{amsmath,amsthm}
\usepackage{amssymb}
\usepackage{multicol}
\usepackage{ulem}
\usepackage{amsfonts}
\usepackage{mathrsfs}
\usepackage[pdfstartview=FitH]{hyperref}
\usepackage{tikz}
\usetikzlibrary{trees,shapes,snakes,arrows,backgrounds}

\allowdisplaybreaks

\newtheorem{lem}{Lemma}[section]

\newtheorem{thm}[lem]{Theorem}
\newtheorem{cor}[lem]{Corollary}

\newtheorem{rem}{Remark}
\theoremstyle{definition}

 \DeclareMathAlphabet{\mathsfsl}{OT1}{cmss}{m}{sl}

 \newcommand{\dif}{\mathrm{d}}

\numberwithin{equation}{section}
\begin{document}

\title{\bf{The Optimal Constant in Generalized Hardy's Inequality}}
\author{\rm\small
\noindent Ying LI\,\,\,\,and\,\,\,\,Yong-Hua MAO\\
\noindent \footnotesize School of Mathematical Sciences, Xiangtan University, \\
\noindent \footnotesize Xiangtan, Hunan 411105, People's Republic of China\\
\noindent \footnotesize School of Mathematical Sciences, Beijing Normal University, \\
\noindent \footnotesize Laboratory of Mathematics and Complex Systems, Ministry of Education\\
\noindent \footnotesize Beijing 100875, People's Republic of China
}
\date{}
 \maketitle
\maketitle \noindent {\centerline{\bf Abstract }}

We obtain the sharp factor of the two-sides estimates of the optimal constant in generalized Hardy's inequality with two general Borel measures on $\mathbb{R}$, which generalizes and unifies the known continuous and discrete cases.
\maketitle

{\bf Keywords and phrases:} Hardy's inequality, optimal constant, Cantor set

{\bf MR(2010) Subject Classification:} 26D10, 26D15, 26A30
\section{Introduction}\label{introd}

The Hardy's inequality is a powerful technical tool not only in advanced theoretical studies of the spectrum of non-negative self-adjoint differential operators such as elliptic operators \cite{davies 1995, rl}, but also in the study of probability such as the stability of diffusion processes or birth-death processes, please refer to \cite[Chapter 6]{chen05} and the references therein. Our motivation is to study the stability of  generalized diffusion processes. However, we shall deal with this problem in separate paper.

For $p>1$ and any non-negative number sequence $\{a_n:\,n\geq 1\}$ such that $\sum_{n=1}^{+\infty} a_n^p<+\infty$, Hardy's inequality was given by
\begin{equation}\label{dishardy}
  \sum_{n=1}^{+\infty}\left(\frac{1}{n}\sum_{ k=1}^n a_k\right)^p\leq \left(\frac{p}{p-1}\right)^p\sum_{n=1}^{+\infty} a_n^p
\end{equation}
in \cite{Hardy}, the optimal constant $\left(\frac{p}{p-1}\right)^p$ was fixed by Laudan, Schur and Hardy in \cite{LSH}.

The continuous analogue of Hardy's inequality (\ref{dishardy}) was introduced in \cite{Hardy} as
\begin{equation}\label{conhardy}
\int_0^{+\infty}\left[\frac{1}{x}\int_0^xf(t)\dif t\right]^p\dif x\leq\left(\frac{p}{p-1}\right)^p\int_0^{+\infty} f(x)^p\dif x
\end{equation}
for $p>1$ and $f\geq0$ such that $f\in L^p(\mathbb{R}^+)$, the optimal constant $\left(\frac{p}{p-1}\right)^p$ was fixed by Hardy in \cite{Hardy2}.

Hardy's inequality has been generalized in various direction. 
In \cite{prok}, Prokhorov gave necessary and sufficient conditions for validity of the Hardy's inequality with three measures.  He also claimed that the Hardy's inequality with three measures can be reduced to the following case with two measures. 
Let $1<p\leq q<+\infty$,~$\mu,\, \nu$ be $\sigma$-finite Borel measures on $\mathbb{R}$, consider
\begin{equation}\label{generalized Hardy}
 \left[ \int_{\mathbb{R}}\left(\int_{(-\infty,x)}f\dif\nu\right)^q\dif\mu(x)\right]^{1/q}
 \leq A\left(\int_{\mathbb{R}} f^p\dif\nu\right)^{1/p}.
\end{equation}
A two-sided estimate for the best constant $A$ can be given as
\begin{equation}\label{kpq 0}
 B\leq A\leq k(q,p)B,
\end{equation}
where the constant $k(q,p)$ can be taken as $p^{1/q}(p^*)^{1/p^*}$ and $B$ is defined in (\ref{B}) below. This findings generalize many existing estimates. For example, please refer to \cite{Brad, muck}  for  both $\mu$ and $\nu$ absolutely continuous with respect to Lebesgue measure and refer to \cite{myh2, miclo} for  both $\mu$ and $\nu$ discrete measures.

When  $\mu$ and $\nu$ are both absolutely continuous with respect to Lebesgue measure,
Maz'ja (\cite{Maz}) presented the factor  $k(q,p)$ as $(q^*)^{1/p^*}q^{1/q}$ for $1<p<q<+\infty$,
Opic and Kufner (\cite{opic}) improved it to $\left(1+q/p^*\right)^{1/q}\left(1+p^*/q\right)^{1/p^*}$ for $1<p\leq q<+\infty$, and Chen (\cite{chen15}) obtained a sharp factor as
\begin{equation}\label{kqp}
   k_{q,p}=\left(\frac{r}{B(1/r,(q-1)/r)}\right)^{1/p-1/q},
 \end{equation}
where $B(a,b)=\int_0^1x^{a-1}(1-x)^{b-1}\dif x$ and $r=q/p-1$.

When $\mu$ and $\nu$ are both discrete measures, Liao (\cite{lzw}) gave the factor $k(q,p)$ as $k_{q,p}$ in (\ref{kqp}) for $1<p\leq q<+\infty$.

A natural question is whether one can also improve the factor  $k(q,p)$ to the sharp $k_{q,p}$ for the above Hardy's inequality (\ref{generalized Hardy}) concerning two general $\sigma$-finite Borel measure? In the present paper, we will give an affirmative answer to this question as follows.
\begin{thm}\label{sharp estimate}
Let $1<p\leq q<+\infty$, $\mu$ and $\nu$ be two $\sigma$-finite Borel measures on $\mathbb{R}$.
Set
 \begin{equation}\label{B}
   B=\sup_{x\in \mathbb{R}}\nu((-\infty,x])^{1/p^*}\mu([x,+\infty))^{1/q}.
 \end{equation}
If $A$ is the optimal constant such that for all $f:\,\mathbb{R}\rightarrow\mathbb{R}$,
\begin{equation}\label{hd2}
  \left[\int_{\mathbb{R}}\left|\int_{(-\infty,x)}f(y)\nu(\dif y)\right|^q\mu(\dif x)\right]^{1/q}\leq A\left[\int_{\mathbb{R}}|f(x)|^p\nu(\dif x)\right]^{1/p},
\end{equation}
then
\begin{equation*}
  B\leq A\leq k_{q,p}B
\end{equation*}
with $k_{q,p}$ defined in $(\ref{kqp})$.
\end{thm}
\begin{rem}\label{the rem of thm}

$(1)$ According to {\rm\cite[p.809]{lzw}}, when $p=q$, the factor $k_{q,p}=p^{1/p}{p^*}^{1/p^*}$, which is consistent with the result in {\rm\cite{chen15,lzw,myh2,miclo,muck,opic,prok}}.

$(2)$ By substituting the interval $(x,+\infty)$ to $(-\infty,x)$ in the left side of $(${\rm\ref{hd2}}$)$, we can get a dual form of Theorem {\rm\ref{sharp estimate}}.

$(3)$ We can also present the sharp factor of the two-side estimate of the optimal constant in the Hardy's inequality with three measures just as in {\rm\cite{prok}}.
\end{rem}

To obtain the sharp factor in (\ref{kqp}), we use the integral transform theorem to explore a new version of Bliss's lemma (see Lemma \ref{generalized bliss}).  Both this new version of Bliss's lemma and its proof are novel as far as we know.

Now, we give some typical examples as applications of the generalized Hardy's inequality in Theorem \ref{sharp estimate}.
In these applications, $\mu$ and $\nu$ can be discrete measures, continuous measures (absolutely continuous w.r.t. Lebesgue measure), and even Cantor measures which are neither continuous nor discrete(see section 3). Additionally, we give the analogue forms as in (\ref{dishardy}) and (\ref{conhardy}) when $p=q$.

\begin{cor}\label{prop1}
Let $\lambda$ denote the standard Bernoulli measure on Cantor set in $[0,+\infty)$. For any non-negative function $f$ and $p>1$, we have
\begin{equation}\label{cantor11}
  \int_0^{+\infty}\left(\frac{1}{\lambda([0,x])}\int_0^x f(t)\lambda(\dif t)\right)^p\lambda(\dif x)
  \leq \left(\frac{p}{p-1}\right)^p\int_0^{+\infty} f(x)^{p}\lambda(\dif x).
\end{equation}
However,  neither the inequality
\begin{equation}\label{cantor12}
  \left[\int_0^{+\infty}\left(\frac{1}{\lambda([0,x])}\int_0^x f(t)\lambda(\dif t)\right)^q\lambda(\dif x)\right]^{1/q}
  \leq A\left[\int_0^{+\infty} f(x)^{p}\lambda(\dif x)\right]^{1/p}
\end{equation} for $1<p<q<+\infty$
nor the inequality
\begin{equation}\label{cantor2}
  \left[\int_0^{+\infty}\left(\frac1x\int_0^xf(t)\dif t\right)^{q}\lambda(\dif x)\right]^{1/q}
  \leq A\left[\int_0^{+\infty} f(x)^{p}\dif x\right]^{1/p}
\end{equation}for $1<p\leq q<+\infty$ holds $($since $A=+\infty)$.
\end{cor}
Observing the proof of Corollary \ref{prop1}, one can get that both (\ref{cantor11}) and (\ref{cantor12}) hold for any $\sigma$-finite Borel measures such that $\Lambda(x):=\lambda([0,x])$ being a continuous increasing function.

By taking one measure discrete and another one absolutely continuous with respect to the Lebesgue measure, we have the  two following mixed forms of Hardy's inequalities.
\begin{cor}\label{mixed forms}
For any non-negative function $f$ and $p>1$, we have
\begin{equation}\label{mixed1}
   \sum_{n=1}^{+\infty}\left(\frac{1}{n}\int_1^n f(t)\dif t\right)^p\leq \left(\frac{p}{p-1}\right)^p\int_1^{+\infty} f^p(x)\dif x.
\end{equation}
And for $1<p\leq q<+\infty$,
  \begin{equation}\label{mixed2}
    \left[\int_1^{+\infty}\left(\frac{1}{x}\sum_{1\leq n< x}f(n)\right)^q\dif x\right]^{1/q}\leq A\left[\sum_{n=1}^{+\infty} f^p(n)\right]^{1/p}
  \end{equation}
holds with $(q-1)^{-1/q}\leq A \leq k_{q,p}(q-1)^{-1/q}$.
\end{cor}


\section{Proof of Theorem \ref{sharp estimate}}

In \cite{chen15, lzw}, a key step in improving the factor to sharp is using the following Bliss lemma \cite{Bliss}  directly or extending it to
the case of discrete measures.
\begin{lem}\label{Bli}
 Let $1<p<q<+\infty$ and $f$ be a non-negative function on $[0,+\infty)$. Then we have
 \begin{equation*}
   \left[\int_0^{+\infty} \dif(-x^{-q/p^*})\left(\int_0^xf(y)\dif y\right)^q\right]^{1/q}\leq k_{q,p}\left[\int_0^{+\infty} f(x)^p\dif x\right]^{1/p}.
 \end{equation*}
Moreover, the optimal constant attains when
 \begin{equation*}
   f(x)=\gamma(\delta x^r+1)^{-(r+1)/r}
 \end{equation*}
with $r=q/p-1$ and $\gamma$, $\delta$ being non-negative constants.
\end{lem}

We will extend Bliss lemma to deal with general Borel measures on $\mathbb{R}$. First, let us recall some basic facts about any Borel measure $\nu$ on $\mathbb{R}$. Define its `cumulative distribution function' and `inverse cumulative distribution function' as:
$$S(x):=\nu((-\infty,x]),\qquad S^{-1}(y):=\inf\{x:S(x)\geq y\}.$$
Since $S$ is right-continuous and increasing, it is well known that
\begin{align}
\{y:S^{-1}(y)\leq x\}=\{y:y\leq S(x)\},&\qquad \{y: S^{-1}(y)>x\}=\{y:y>S(x)\},\label{inverse relation}\\
S(S^{-1}(y))\geq y,&\qquad S(S^{-1}(y)-)\leq y. \label{continuous relation}
\end{align}
In particular, if $S$ is continuous, then $S(S^{-1}(y))=y$.

Let $m$ denote the Lebesgue measure, for any $-\infty<a<b<+\infty$, we have from (\ref{inverse relation}) that
\begin{eqnarray*}
 m_{S^{-1}}((a,b])&:=&m(\{t:S^{-1}(t)\in(a,b]\})=m(\{t:t\in(S(a),S(b)]\})\\
  &=&\int_{S(a)}^{S(b)}\dif t=S(b)-S(a)=\nu((a,b]).
\end{eqnarray*}
Then the measure extension theorem implies that $m_{S^{-1}}=\nu$.

According to the integral transform theorem (see for example \cite[Theorem 39.C.]{halmos}), for any Borel set $\Gamma$ and measurable function $f$, it follows that
\begin{equation}\label{integral exchange2}
  \int_{\Gamma}f\dif\nu=\int_{\{y:\,S^{-1}(y)\in\Gamma\}}f\circ S^{-1}(y)\dif y.
\end{equation}

Now, we state our generalized Bliss lemma.
\begin{lem}\label{generalized bliss}
 Let $k_{q,p},\,S(x)$ prevail. For any $x\in\mathbb{R}$, the Borel measure $\widetilde{\nu}$ is defined by $\widetilde{\nu}((x,+\infty)):=S(x)^{-q/p^*}$.
If $S(+\infty)=+\infty$, then for any non-negative real function $f$ and $1<p\leq q<+\infty$, we have
 \begin{equation*}
   \left[\int_{\mathbb{R}}\widetilde{\nu}(\dif x)\left(\int_{(-\infty,x)}f(y)\nu(\dif y)\right)^q\right]^{1/q}
   \leq k_{q,p}\left[\int_{\mathbb{R}}f(y)^p\nu(\dif y)\right]^{1/p}.
 \end{equation*}
\end{lem}
\begin{proof}
In the case of $p=q$, the assertion holds as a result of Remark~\ref{the rem of thm}.(1) and \cite[Theorem 1]{prok}.

In the case of $p<q$, set $\widetilde{m}(\dif x)=\dif(-x^{-q/p^*})$. Since $S(+\infty)=+\infty$, we have that  for any $x\in\mathbb{R}$,
\begin{eqnarray*}
 \widetilde{m}_{S^{-1}}((x,+\infty))&:=&\widetilde{m}(\{t:S^{-1}(t)\in(x,+\infty)\})=\widetilde{m}(\{t: t\in(S(x),+\infty)\})\\
  &=&\int_{S(x)}^{+\infty}\dif(-t^{-q/p^*})=S(x)^{-q/p^*}=\widetilde{\nu}((x,+\infty)).
\end{eqnarray*}
Then we have $\widetilde{m}_{S^{-1}}=\widetilde{\nu}$ by measure extension theorem. Moreover, the integral transform formula implies that for any measurable function $g$,
 \begin{equation}\label{integral exchange}
   \int_{\mathbb{R}}g(x)\widetilde{\nu}(\dif x)=\int_0^{+\infty} g\circ S^{-1}(x)\dif(-x^{-q/p^*}).
 \end{equation}
By the dominated convergence theorem, (\ref{inverse relation}) and (\ref{integral exchange2}), we have
\begin{eqnarray*}
  \int_{(-\infty,u)}f(y)\nu(\dif y)&=&\int_{\mathbb{R}}I_{(-\infty,u)}(y)f(y)\nu(\dif y)\\
  &=&\int_{\mathbb{R}}\lim_{u_n\uparrow u}I_{(-\infty,u_n]}(y)f(y)\nu(\dif y)
  =\lim_{u_n\uparrow u}\int_{(-\infty,u_n]}f(y)\nu(\dif y)\\
  &=&\lim_{u_n\uparrow u}\int_{(0,S(u_n)]}f\circ S^{-1}(y)\dif y
  \leq\int_{(0,S(u-)]}f\circ S^{-1}(y)\dif y.
\end{eqnarray*}
Furthermore, substituting $g(x)=\left(\int_{(-\infty,x)}f\dif\nu\right)^q$ into (\ref{integral exchange}),
we obtain from (\ref{continuous relation}) that
\begin{eqnarray*}
 \int_{\mathbb{R}}\widetilde{\nu}(\dif x)\left(\int_{(-\infty,x)}f(y)\nu(\dif y)\right)^q
&=&\int_0^{+\infty}\dif(-x^{-q/p^*})\left(\int_{(-\infty,S^{-1}(x))}f(y)\nu(\dif y)\right)^q\\
&\leq&\int_0^{+\infty}\dif(-x^{-q/p^*})\left(\int_{(0,S(S^{-1}(x)-)]}f\circ S^{-1}(y)\dif y\right)^q\\
&\leq&\int_0^{+\infty}\dif(-x^{-q/p^*})\left(\int_{(0,x]}f\circ S^{-1}(y)\dif y\right)^q.
\end{eqnarray*}
According to Lemma \ref{Bli} and (\ref{integral exchange2}), we have
\begin{eqnarray*}
\left[\int_0^{+\infty}\dif(-x^{-q/p^*})\left(\int_0^xf\circ S^{-1}(y)\dif y\right)^q\right]^{1/q}
&\leq&k_{q,p}\left(\int_0^{+\infty}\left(f\circ S^{-1}(x)\right)^p\dif x\right)^{1/p}\\
&=&k_{q,p}\left(\int_{\mathbb{R}}f(y)^p\nu(\dif y)\right)^{1/p}.
\end{eqnarray*}
\end{proof}
The next technical lemma shows that if one measure is dominated by another measure, then so does their integrations.\begin{lem}\label{compare}
  Let $\mu_1$ and $\mu_2$ be two $\sigma$-finite Borel measures. If
  \begin{equation*}
  \mu_1((x,+\infty))\leq\mu_2((x,+\infty)),\quad \forall~x\in\mathbb{R},
   \end{equation*}
 then for any non-negative increasing function $f$, we have
  \begin{equation*}
  \int_{\mathbb{R}}f(x)\mu_1(\dif x)\leq\int_{\mathbb{R}}f(x)\mu_2(\dif x).
  \end{equation*}
\end{lem}
\begin{proof}
According to Fubini theorem, for any non-negative increasing function $f$ and $\sigma$-finite measures $\mu_i\,(i=1,2)$,
\begin{eqnarray*}
  \int f(x)\mu_i(\dif x)&=&\int_{\{x:f(x)>0\}}f(x)\mu_i(\dif x)\\
  &=&\int_{\{x:f(x)>0\}}\mu_i(\dif x)\int_0^{f(x)}\dif t\\
  &=&\int_{\mathbb{R}} I_{\{x:f(x)>0\}}\mu_i(\dif x)\int_0^{+\infty} I_{\{t:f(x)>t\}}\dif t\\
  &=&\int_0^{+\infty}\dif t\int_{\mathbb{R}} I_{\{x:f(x)>t\}}\mu_i(\dif x)\\
  &=&\int_0^{+\infty}\mu_i(\{x:f(x)>t\})\dif t.
\end{eqnarray*}
Since $f$ is an increasing function, it is easy to check that for any given $t\geq0$, the set $\{x:f(x)>t\}$ have the form of $(a,+\infty)$ or $[a,+\infty)$.
Thus, it suffices to show that
  \begin{equation}\label{dominated2}
  \mu_1([x,+\infty))\leq\mu_2([x,+\infty)),\,\,x\in\mathbb{R}.
 \end{equation}
Without loss of generality, suppose for any given $x\in\mathbb{R}$, $\mu_2((x,+\infty))<+\infty$. Since $\mu_2$ is a Borel measure, we have $\mu_2((x-1/n,x])<+\infty$ for any $n\geq 1$. Furthermore,
\begin{equation*}
  \mu_1((x-1/n,+\infty))\leq\mu_2((x-1/n,+\infty))=\mu_2((x-1/n,x])+\mu_2((x,+\infty))<+\infty.
\end{equation*}
Then (\ref{dominated2}) holds by the upper continuity of $\mu_i(i=1,2)$.
\end{proof}
\begin{proof}[Proof of Theorem \ref{sharp estimate}:] We divide the proof into two steps:

(i) First, we prove the first assertion provided $\nu(\mathbb{R})=+\infty$.
To avoid the trivial case, assume $B<+\infty$. Let
\begin{equation*}
S(x)=\nu((-\infty,x]),\,\,\,\,\widetilde{\nu}((x,+\infty))=S(x)^{-q/p^*}.
\end{equation*}
By the definition of $B$, we have that for any $x\in\mathbb{R}$,
\begin{equation*}
\mu((x,+\infty))\leq\mu([x,+\infty))\leq B^q \nu((-\infty,x])^{-q/p^*}=B^qS(x)^{-q/p^*}=B^q\widetilde{\nu}((x,+\infty)).
\end{equation*}
According to Lemma \ref{compare} and Lemma \ref{generalized bliss}, for any non-negative function $f$, we have
\begin{eqnarray*}
  \int_{\mathbb{R}}\mu(\dif x)\left(\int_{(-\infty,x)}f(y)\nu(\dif y)\right)^q
  &\leq&B^q\int_{\mathbb{R}}\widetilde{\nu}(\dif x)\left(\int_{(-\infty,x)}f(y)\nu(\dif y)\right)^q\\
  &\leq& k_{q,p}^q B^q\left(\int_{\mathbb{R}}f(x)^p\nu(\dif x)\right)^{q/p}.
\end{eqnarray*}
Thus, $A\leq k_{q,p}B$. In addition, we have $B\leq A$ according to \cite[Theorem 1]{prok}. Hence, $B\leq A\leq k_{q,p}B$.

(ii) The next step is to remove the condition $\nu(\mathbb{R})=+\infty$. This is easy to overcome by \cite[Lemma 4.2]{chen15}.
\end{proof}
\section{Proof of Corollaries \ref{prop1} and \ref{mixed forms}}
First, we recall the standard Bernoulli measure on Cantor set in $\mathbb{R}$.
Let $\Omega_i=\{0,1\},~i=0,1,\cdots$, and $\rho_m$ be the uniform probability measure on $\Omega^m:=\prod_{i=0}^m\Omega_i$,
that is $\rho_m(\{x\})=2^{-(m+1)}$ for any $(x_0,x_1,\cdots,x_m)\in\Omega^m$. Consider the map $J:\,\Omega^m\rightarrow[0,1]$,
\begin{equation*}
  \forall\,x=(x_0,x_1,\cdots,x_m)\in\Omega^m,
  \,\,\,\,J(x):=a_0^mx_0+a_1^mx_1+\cdots+a_m^mx_m,
\end{equation*}
where $a_k^m=3^{-m}b_k,\,b_0=1,\,b_k=2\cdot3^{k-1}$.

Let $K_m=J(\Omega^m)$. Then the closure of $\cup_{m=0}^{+\infty} K_m$ is Cantor set in $[0,1]$, denoted by $\mathbb{K}$.
Let $\lambda_m=\rho_m\circ J^{-1}$, then $\lambda_m(\{p\})=2^{-(m+1)},\,\forall~p\in K_m$.

Following \cite{Kigami}, we know that there exists a unique probability measure $\lambda$ on $\mathbb{K}$ such that
$\lambda_m\Rightarrow\lambda$, that is, $\forall~f\in C(\mathbb{K}),\,\,\lim_{m\rightarrow+\infty}\int_{K_m}f\dif\lambda_m=\int_{\mathbb{K}}f\dif\lambda$,
thus $\lambda$ is called the standard Bernoulli (probability) measure on $\mathbb{K}$. Let $\widetilde{\mathbb{K}}=\cup_{n=0}^{+\infty} (n+\mathbb{K})$
be Cantor set on $[0,+\infty)$ and denote again by $\lambda$ the extended Bernoulli measure on $\widetilde{\mathbb{K}}$.

Under our settings, we can have an analogue of Hardy's inequality on Cantor set, see Proposition \ref{prop1} in section \ref{introd}. Now, we give the proof of these results.
\begin{proof}[Proof of Proposition \ref{prop1}:]
We need to calculate $B$ in (\ref{B}).

(i) To prove (\ref{cantor11}) and (\ref{cantor12}), set $\nu=\lambda$ and $\mu(\dif x)=\lambda([0,x])^{-q}\lambda(\dif x)$ on $[0,+\infty)$. Clearly,
\begin{equation*}
  B=\sup_{x\in[0,+\infty)}\lambda([0,x])^{1/p^*}\left(\int_x^{+\infty}\frac{\lambda(\dif t)}{\lambda([0,t])^q}\right)^{1/q}.
\end{equation*}
Let $\Lambda(x)=\lambda([0,x])$. Then $\Lambda$ is an increasing continuous function and $\Lambda(+\infty)=+\infty$.
Define $\Lambda^{-1}(y)=\inf\{x:\Lambda(x)\geq y\}$, it is easy to see that $\Lambda(\Lambda^{-1}(y))=y$.
The integral transform formula implies that for any Borel measurable function $g$
\begin{equation*}
  \int_x^{+\infty} g(\Lambda(t))\lambda(\dif t)=\int_{\Lambda(x)}^{\Lambda(+\infty)}g(\Lambda(\Lambda^{-1}(t)))\dif t
  =\int_{\Lambda(x)}^{+\infty} g(t)\dif t.
\end{equation*}
Take $g(t)=t^{-q}$ in this identity, we get
\begin{eqnarray*}
 \int_x^{+\infty}\frac{\lambda(\dif t)}{\lambda([0,t])^q}&=&\int_x^{+\infty}\Lambda(t)^{-q}\lambda(\dif t)=\int_{\Lambda(x)}^{+\infty} t^{-q}\dif t\\
 &=&(q-1)^{-1}\Lambda(x)^{1-q}=(q-1)^{-1}\lambda([0,x])^{1-q}.
\end{eqnarray*}
Since $p\leq q$, we have $p^*\geq q^*$. Hence,
\begin{equation*}
  B=(q-1)^{-1/q}\sup_{x\in[0,+\infty)}\lambda([0,x])^{1/p^*-1/q^*}=
         \left\{
      \begin{array}{ll}
      (q-1)^{-1/q}, &p=q,\\
      +\infty, &\text{otherwise}.
      \end{array}
\right.
\end{equation*}
In the case of $p=q$, we have from Remark~\ref{the rem of thm}.(1) and $p/(p-1)=p^*$ that
\begin{align*}
  k_{p,p}B=p^{1/p}{p^*}^{1/p^*}(p-1)^{-1/p}=\left(\frac{p}{p-1}\right)^{1/p}\cdot\left(\frac{p}{p-1}\right)^{1/p^*}=\frac{p}{p-1}.
\end{align*}

(ii) To prove (\ref{cantor2}), let $\mu(\dif x)=x^{-q}\lambda(\dif x)$ and $\nu$ be the Lebesgue measure on $[0,+\infty)$. It is obvious that
\begin{equation*}
 B=\sup_{x\in[0,+\infty)}x^{1/p^*}\left(\int_x^{+\infty} t^{-q}\lambda(\dif t)\right)^{1/q}=\sup_{x\in\widetilde{\mathbb{K}}}x^{1/p^*}\left(\int_x^{+\infty} t^{-q}\lambda(\dif t)\right)^{1/q}.
\end{equation*}
Take $x_m=3^{-m}$ to derive
\begin{eqnarray*}
  (x_m)^{1/p^*}\left(\int_{[x_m,+\infty)}t^{-q}\lambda(\dif t)\right)^{1/q}
  &=&\left(3^{-m}\right)^{1/p^*}\left(\int_{[3^{-m},+\infty)}t^{-q}\lambda(\dif t)\right)^{1/q}\\
  &\geq&\left(3^{-m}\right)^{1/p^*}\left[\left(3^{-m}\right)^{-q}\cdot2^{-(m+1)}\right]^{1/q}\\
  &=& 2^{-1/q}\cdot\left(\frac{3^{1/p}}{2^{1/q}}\right)^m.
\end{eqnarray*}
Since $1<p\leq q<+\infty$, we have $3^{1/p}\geq 3^{1/q}>2^{1/q}$. Then we get
\begin{equation*}
  2^{-1/q}\cdot\left(\frac{3^{1/p}}{2^{1/q}}\right)^m \longrightarrow +\infty,\,\,\,\,\text{if } m\rightarrow +\infty.
\end{equation*}
Consequently, we have $B=+\infty$.
\end{proof}
\begin{proof}[Proof of Corollary \ref{mixed forms}:]
To prove (\ref{mixed1}), let $\mu$ be the measure on $\mathbb{N}$ with $\mu_n=n^{-p}$ and $\nu(\dif x)=\dif x$ on $[1,+\infty)$. Clearly,
\begin{equation*}
  B=\sup_{x\geq1}(x-1)^{1/p^*}\left(\sum_{k\geq x}k^{-p}\right)^{1/p}=\sup_{n\geq1}(n-1)^{1/p^*}\left(\sum_{k\geq n}k^{-p}\right)^{1/p}.
\end{equation*}
For any $n\in\mathbb{Z}^+$, on the one hand,
\begin{align*}
   \sum_{k\geq n}k^{-p}=\sum_{k\geq n}\int_{k-1}^kk^{-p}\dif x\leq\sum_{k\geq n}\int_{k-1}^kx^{-p}\dif x=\frac{(n-1)^{1-p}}{p-1}.
\end{align*}
Then we have
\begin{align*}
  B\leq(p-1)^{-1/p}\sup_{n\geq1}(n-1)^{1/p^*-1/p^*}=(p-1)^{-1/p}.
\end{align*}
On the other hand,
\begin{align*}
   \sum_{k\geq n}k^{-p}=\sum_{k\geq n}\int_k^{k+1}k^{-p}\dif x\geq\sum_{k\geq n}\int_k^{k+1}x^{-p}\dif x=\frac{n^{1-p}}{p-1}.
\end{align*}
Then we have
\begin{align*}
  B\geq(p-1)^{-1/p}\sup_{n\geq1}\left(\frac{n-1}{n}\right)^{1/p^*}=(p-1)^{-1/p}.
\end{align*}
Consequently,
\begin{equation*}
  B=(p-1)^{-1/p}.
\end{equation*}
Then we have $k_{p,p}B=\frac{p}{p-1}$.

To prove (\ref{mixed2}), let $\nu$ be the counting measure on $\mathbb{N}$ and $\dif\mu(x)=x^{-q}\dif x$ on $[1,+\infty)$,
we get
\begin{equation*}
  B=\sup_{x\geq 1}[x]^{1/p^*}\left(\int_x^{+\infty} t^{-q}\dif t\right)^{1/q}=(q-1)^{-1/q}\sup_{x\geq 1}[x]^{1/p^*}x^{-1/q^*}=(q-1)^{-1/q}.
\end{equation*}
\end{proof}
\bigskip

{\bf Acknowledgements}\    ~Research supported in part by 985 Project, NSFC(No 11131003, 11501531, 11571043), SRFDP(No 20100003110005) and the Fundamental Research Funds for the Central Universities.



\end{document}